\documentclass{article}

\usepackage{srcltx}
\usepackage{stmaryrd}
\usepackage[all]{xy}
\usepackage{amssymb,latexsym}
\usepackage{graphicx}
\usepackage{amsfonts}
\usepackage[T1]{fonte nc}
\usepackage[french,english]{babel}
\usepackage{amssymb}
\usepackage{mathrsfs}
\usepackage{stmaryrd}
\usepackage{url}
\usepackage{amsmath}
\usepackage{wasysym}
\usepackage[all]{xy}
\usepackage{fancyhdr}
\usepackage{epic}
\usepackage{enumerate}

\newtheorem{theorem}{Theorem}[section]
\newtheorem{corollary}[theorem]{Corollary}
\newtheorem{lemma}[theorem]{Lemma}
\newtheorem{proposition}[theorem]{Proposition}
\newtheorem{conjecture}[theorem]{Conjecture}

\newtheorem{remark}[theorem]{Remark}

\newenvironment{proof}[1][Proof]{\begin{trivlist}
\item[\hskip \labelsep {\bfseries #1}]}{\end{trivlist}}

\def\N{{\mathbb N}}
\def\R{{\mathbb R}}
\def\Z{{\mathbb Z}}
\def\C{{\mathbb C}}

\def\gg{\mathfrak g}
\def\eg{\emph{e.g.\,}}

\def\g{\mathfrak{g}}
\def\h{\mathfrak{h}}

\def\hg{\widehat{\mathfrak g}}

\def\la{\lambda}

\def\L{\Lambda}

\def\<{\langle\,}
\def\>{\,\rangle}

\def\ds{\displaystyle}

\def\L{\mathfrak{L}}

\title{On the KNS conjecture in type $E$}

\author{Anne-Sophie Gleitz}

\date{}

\begin{document}

\maketitle

\begin{abstract}
For the exceptional types $E_6$, $E_7$, and $E_8$,
we prove that the specializations at roots of unity of the quantum dimensions of the 
Kirillov-Reshetikhin modules give real solutions of $\ell$-restricted $Q$-systems,
as conjectured by Kuniba, Nakanishi, and Suzuki \cite{KNS1}.
We also show that these solutions are positive in type $E_6$. In type $E_7$ and $E_8$, 
we only prove positivity for a subset of the nodes of the Dynkin diagram.
\end{abstract}

\section{Introduction}

Let $\delta$ be an ADE Dynkin diagram with vertex set $I$. 
Kirillov and Reshetikhin~\cite{KR} have attached to $\delta$ an infinite system of algebraic equations
called a $Q$-system, with unknowns $Q^{(i)}_k\ (i\in I,\ k\in \N)$ in a commutative ring.
It is given by
\begin{equation}\label{eq1}
\left(Q_k^{(i)}\right)^2 = Q_{k-1}^{(i)} Q_{k+1}^{(i)} + \prod_{j\sim i} Q_k^{(j)},
\qquad
(i\in I,\ k\ge 1),
\end{equation}
where, in the product, the notation $j\sim i$ means that $j$ runs over all neighbours 
of $i$ in the diagram $\delta$. One usually imposes the initial condition: 
\begin{equation}\label{eq2}
Q_0^{(i)} = 1,\qquad (i\in I). 
\end{equation}
Fix a positive integer $\ell$. The $\ell$-restricted $Q$-system has a finite number of
unknowns $Q^{(i)}_k\ (i\in I,\ 0\le k\le \ell)$ satisfying the same equations (\ref{eq1}) and (\ref{eq2})
together with the new boundary condition:
\begin{equation}
Q_\ell^{(i)} = 1,\qquad (i\in I). 
\end{equation}

It is known (see \cite{Lee}) that $\ell$-restricted $Q$-systems always have a unique \emph{positive} solution,
that is, a unique solution satisfying 
\begin{equation}
 Q_k^{(i)} \in \R_{>0}, \qquad (i\in I,\ 0\le k\le \ell).
\end{equation}
There are various motivations coming from conformal field theory and algebraic $K$-theory
to obtain an explicit description of this positive solution \cite{KNS2}. In particular, 
if $Q^{(i)}_k\ (i\in I,\ 0\le k\le \ell)$ is this solution, then
the numbers
\[
x^{(i)}_k := \frac{\prod_{j\sim i} Q_k^{(j)}}{(Q_k^{(i)})^2} \in ]0,1[, 
\qquad (i\in I,\ 0\le k\le \ell),
\]
are the arguments of a nice dilogarithm identity, see \cite[Th. 5.2, Prop. 14.1]{KNS2}.

Kuniba, Nakanishi and Suzuki \cite{KNS1} have given a conjectural description of the
positive solution of an $\ell$-restricted $Q$-system in terms of quantum dimensions of Kirillov-Reshetikhin
modules over the quantum affine algebra attached to $\delta$ (see below Conjecture~\ref{KNS-conj} for a precise statement).
The conjecture is easy to check in type $A$, but it remained an open problem
for a long time in other types. Recently, Lee \cite{Lee} has given a proof 
of the KNS conjecture in type $D$. 
In this paper we study the KNS conjecture in type $E$.
We show (Theorem~\ref{knsE678}) that the quantum dimensions of Kirillov-Reshetikhin modules give indeed real
solutions of the $\ell$-restricted $Q$-system.
We also prove the positivity of these solutions for type $E_6$,
and for some subsets of nodes of the Dynkin diagrams in type $E_7$ and $E_8$.
The main ingredients of our proofs are (i) some character formulas of Chari for Kirillov-Reshetikhin 
modules attached to certain nodes of the Dynkin diagrams, (ii) the level $\ell$ 
action of the affine Weyl group on the weight lattice, already used by Lee in \cite{Lee},
and (iii) the notion of log-concavity for sequences of positive real numbers.

What makes type $E$ more difficult than type $D$ is that there are many nodes of the Dynkin
diagram with labels $a_i > 2$, so it is difficult to check positivity of the quantum dimensions
at these nodes. Moreover, there are no convenient characters formulas in the literature for the 
Kirillov-Reshetikhin modules attached to these nodes, so in order to get
around this difficulty, we first check the KNS conjecture for the extremal nodes, and then
move to the remaining nodes using the $Q$-system.

\subsection*{Acknowledgements}
The author would like to thank B. Leclerc for his help, precious advice and patience, 
D. Hernandez for providing useful information on Kirillov-Reshe\-ti\-khin modules,
P. Br\"and\'en for some correspondence about log-concavity, and
P.~Browne for kindly providing his tables of positive roots.

\section{Statement of the KNS conjecture and of the main results}

\subsection{Simple Lie algebras and root systems}
Let $\g$ be a simple Lie algebra over $\C$ with Dynkin diagram $\delta$.
Fix a Cartan subalgebra~$\h$.
Let $\Phi\subset \h^*$ be the root system of $\g$, and $\Phi^+$ its subset of positive roots.
Let $\alpha_i, \varpi_i\ (i\in I)$ denote the simple roots and the fundamental 
weights, respectively. 
We fix a symmetric bilinear form $(\cdot\mid\cdot)$ on $\h^*$, normalized
by
\begin{equation}
 (\alpha_i\mid\alpha_i) = 2,\qquad  (\alpha_i \mid \varpi_j) = \delta_{ij}.
\end{equation}
We then have 
\begin{equation}
\alpha_i = \sum_{j\in I} c_{ji}\varpi_j, 
\end{equation}
where the coefficients $c_{ji} = (\alpha_j\mid\alpha_i)$ 
are the entries of the Cartan matrix of $\g$.
We denote by 
\begin{equation}
\rho = \sum_{i\in I} \varpi_i,
\qquad 
\theta = \sum_i a_i \alpha_i,
\end{equation}
the Weyl vector and the highest root, respectively. Here the $a_i\ (i\in I)$
are the Dynkin labels. 
We have \cite[Chap.6, Prop. 3.8]{B}
\begin{equation}\label{eq-Coxeter}
(\rho\mid \theta) = \sum_{i\in I} a_i = h-1, 
\end{equation}
where $h$ is the Coxeter number.

Let $P=\bigoplus_{i\in I} \Z\varpi_i$ be the weight lattice, and 
$P_+=\bigoplus_{i\in I} \N\varpi_i$ the monoid of dominant weights.
For $\lambda\in P_+$, we denote by $V(\lambda)$ the irreducible complex representation of $\g$
with highest weight~$\lambda$, and by $\chi(V(\lambda))\in\Z[P]$ its character. 
The dimension of $V(\lambda)$ is given by Weyl's formula:
\begin{equation}
 \dim V(\lambda) = \prod_{\beta\in\Phi^+} \frac{(\lambda+\rho\mid \beta)}{(\rho\mid\beta)}.
\end{equation}
For $\zeta\in\C^*$, and $k\in\Z$ we define the $\zeta$-integer $[k]_\zeta := (\zeta^k-\zeta^{-k})(\zeta-\zeta^{-1})^{-1}$.
The \emph{$\zeta$-dimension} of $V(\lambda)$ is the $\zeta$-analogue of $\dim V(\lambda)$ given by
\begin{equation}\label{eq-zeta}
d_\zeta(\lambda) := \prod_{\beta\in\Phi^+} \frac{[(\lambda+\rho\mid \beta)]_\zeta}{[(\rho\mid\beta)]_\zeta}. 
\end{equation}
When $|\zeta|=1$, this is a real number, well-defined if $[(\rho\mid\beta)]_\zeta \not = 0$ for ev\nolinebreak ery \nolinebreak$\beta\in\nolinebreak \Phi^+$.

\subsection{Kirillov-Reshetikhin modules}

Let $U_q(\hg)$ be the Drinfeld-Jimbo quantum enveloping algebra of the affine Lie algebra
$\hg$ associated with $\g$ (see \eg \cite{CP}).
It has a natural subalgebra isomorphic to the quantum enveloping algebra $U_q(\g)$ of $\g$.
Here, we assume that the quantum parameter $q\in\C^*$ is not a root of unity.
It follows that we can identify the characters of $U_q(\g)$ with those of $\g$.

The Kirillov-Reshetikhin modules $W^{(i)}_{k,a}$ are some special irreducible finite-dimensional representations
of $U_q(\hg)$, depending on three parameters $i\in I$, $k\in\N$, $a\in \C^*$ (see \cite[\S4.2]{KNS2}).
We will only be interested in their restrictions to $U_q(\g)$, which are independent of $a$,
and will be denoted by $W^{(i)}_k$. The $U_q(\g)$-module $W^{(i)}_k$ is not irreducible
in general. Its character can be expressed as an $\N$-linear combination of irreducible characters of $\g$:
\begin{equation}\label{krm}
 \chi\left(W^{(i)}_k\right) = \sum_{\lambda\in P_+} a^{(i)}_k(\lambda)\, \chi(V(\lambda)).
\end{equation}
It follows from the Kirillov-Reshetikhin conjecture, proved by Nakajima \cite{N}, 
that the characters $\chi(W^{(i)}_k)\ (i\in I,\,k\in\N)$
give a solution of the unrestricted $Q$-system (\ref{eq1}) (\ref{eq2}) in the ring $\Z[P]$,
see \cite[Prop. 13.10]{KNS2}.

Define the $\zeta$-dimension of a Kirillov-Reshetikhin module by
\begin{equation}\label{decomp}
 d_\zeta\left(W^{(i)}_k\right) := \sum_{\lambda\in P_+} a^{(i)}_k(\lambda)\, d_\zeta(\lambda).
\end{equation}
Since the map $\chi(V) \mapsto d_\zeta(V)$ is additive and multiplicative, the numbers
$d_\zeta(W^{(i)}_k)\ (i\in I,\,k\in\N)$ give a solution of the unrestricted $Q$-system in $\C\:$.

\subsection{The KNS conjecture}

From now on, we fix $\ell \in \Z_{>0}$ and we set 
\begin{equation}\label{not}
l := \ell + h,\qquad  \zeta := \exp(i\pi/l).
\end{equation}
Because of (\ref{eq-Coxeter}), the $\zeta$-dimension (\ref{eq-zeta}) is well-defined for every $\lambda\in P_+$.

\begin{conjecture}[{\cite[Conjecture 14.2]{KNS1}}]\label{KNS-conj}
The collection of real numbers 
\[
Q^{(i)}_k := d_\zeta\left(W^{(i)}_k\right),\qquad (i\in I,\ 0\le k \le \ell),
\]
is a solution of the $\ell$-restricted $Q$-system. Moreover
the following properties hold for any $i\in I$:
\begin{itemize}
 \item[(i)] $d_\zeta\left(W^{(i)}_{k}\right) = 0$ for $k\in  \llbracket \ell + 1,l-1\rrbracket$.
 \item[(ii)] $Q_k^{(i)} = Q^{(i)}_{\ell-k}$ for $k\in\llbracket 0,\ell\rrbracket$.
 \item[(iii)] $Q^{(i)}_k > 0$ for $k\in\llbracket 0,\ell\rrbracket$. 
 \item[(iv)] $Q_k^{(i)} < Q_{k+1}^{(i)}$ for $k\in \llbracket 0,\lfloor \ell/2\rfloor-1\rrbracket$.
\end{itemize}
\end{conjecture}

\subsection{Main result}
Let $\g$ be of exceptional type $E$.
We follow the convention of \cite{B} for numbering the vertices of its Dynkin diagram $\delta$
(see Figure~\ref{fig:Dynkin} below).  
Recall that $h=12$ in type $E_6$, $h=18$ in type $E_7$, and $h=30$ in type $E_8$.
The aim of this paper is to prove the following: 
\begin{theorem}\label{knsE678}
The collection of real numbers
\[
Q^{(i)}_k := d_\zeta\left(W^{(i)}_k\right),\qquad (i\in I,\ 0\le k \le \ell),
\]
is a solution of the $\ell$-restricted $Q$-system. Moreover:
\begin{itemize}
 \item[(a)] In type $E_6$, properties (i), (ii), (iii), (iv) of Conjecture \ref{KNS-conj} hold for any $i\in I$.
We also have
\[
d_\zeta\left(W^{(i)}_{k+l}\right) = d_\zeta\left(W^{(i)}_k\right), \qquad (k\in\N).
\]
 \item[(b)] In type $E_7$, properties (i), (ii) of Conjecture \ref{KNS-conj} hold for any $i\in I$,  
property (iii) holds for $i= 1,2,3,6,7$, and property (iv) holds for $i=1,2,7$.
We also have 
\[
\begin{array}{cccl}
d_\zeta\left(W^{(i)}_{k+l}\right) &=& d_\zeta\left(W^{(i)}_k\right), & (k\in\N,\ i=1,3,4,6),\\[3mm]
d_\zeta\left(W^{(i)}_{k+l}\right) &=& -d_\zeta\left(W^{(i)}_k\right),& (k\in\N,\ i = 2,5,7).
\end{array}
\]
 \item[(c)] In type $E_8$, property (i) of Conjecture \ref{KNS-conj} holds for any $i\in I$,
property (ii) holds for $i= 1,3,4,5,6,7,8$, property (iii) holds for $i= 1,3,8$,
and property (iv) holds for $i=1,8$.
We also have for any $i\in I$
\[
d_\zeta\left(W^{(i)}_{k+l}\right) = d_\zeta\left(W^{(i)}_k\right), \qquad (k\in\N).
\]
\end{itemize}
\end{theorem}
The proof of Theorem~\ref{knsE678} will be given in \S\ref{proofE6}, \S\ref{proofE7}
and \S\ref{proofE8}. 

\begin{remark}
{\rm
Extensive computer calculations give evidence that the missing properties 
in types $E_7$ and $E_8$ also hold. To prove the positivity property at all nodes
in type $E_7$ and $E_8$, we would have to show that certain explicit sequences 
of real algebraic numbers are $k$-log-concave for some small values of $k$
(see below Remark~\ref{remark-inf-logc} for type $E_7$). 
In fact we conjecture that these sequences are $\infty$-log-concave.
}
\end{remark}

\begin{figure}[!t]
\begin{picture}(40,150)(-40,-200)
\put(0,-50){
\put(0,0){\circle{6}}
\put(20,0){\circle{6}}
\put(40,0){\circle{6}}
\put(60,0){\circle{6}}
\put(80,0){\circle{6}}
\put(40,20){\circle{6}}
\drawline(3,0)(17,0)
\drawline(23,0)(37,0)
\drawline(43,0)(57,0)
\drawline(40,3)(40,17)
\drawline(63,0)(77,0)
\put(-30,-2){$E_6$}
\put(-2,-15){\small $1$}
\put(18,-15){\small $3$}
\put(38,-15){\small $4$}
\put(58,-15){\small $5$}
\put(78,-15){\small $6$}
\put(50,18){\small $2$}
}
\put(0,-110){
\put(0,0){\circle{6}}
\put(20,0){\circle{6}}
\put(40,0){\circle{6}}
\put(60,0){\circle{6}}
\put(80,0){\circle{6}}
\put(100,0){\circle{6}}
\put(40,20){\circle{6}}
\drawline(3,0)(17,0)
\drawline(23,0)(37,0)
\drawline(43,0)(57,0)
\drawline(40,3)(40,17)
\drawline(63,0)(77,0)
\drawline(83,0)(97,0)
\put(-30,-2){$E_7$}
\put(-2,-15){\small $1$}
\put(18,-15){\small $3$}
\put(38,-15){\small $4$}
\put(58,-15){\small $5$}
\put(78,-15){\small $6$}
\put(98,-15){\small $7$}
\put(50,18){\small $2$}
}
\put(0,-170){
\put(0,0){\circle{6}}
\put(20,0){\circle{6}}
\put(40,0){\circle{6}}
\put(60,0){\circle{6}}
\put(80,0){\circle{6}}
\put(100,0){\circle{6}}
\put(120,0){\circle{6}}
\put(40,20){\circle{6}}
\drawline(3,0)(17,0)
\drawline(23,0)(37,0)
\drawline(43,0)(57,0)
\drawline(40,3)(40,17)
\drawline(63,0)(77,0)
\drawline(83,0)(97,0)
\drawline(103,0)(117,0)
\put(-30,-2){$E_8$}
\put(-2,-15){\small $1$}
\put(18,-15){\small $3$}
\put(38,-15){\small $4$}
\put(58,-15){\small $5$}
\put(78,-15){\small $6$}
\put(98,-15){\small $7$}
\put(118,-15){\small $8$}
\put(50,18){\small $2$}
}
\end{picture}
\caption{The Dynkin diagrams for ${\mathfrak g}=E_6,E_7,E_8$ and their node numberings}
\label{fig:Dynkin}
\end{figure}
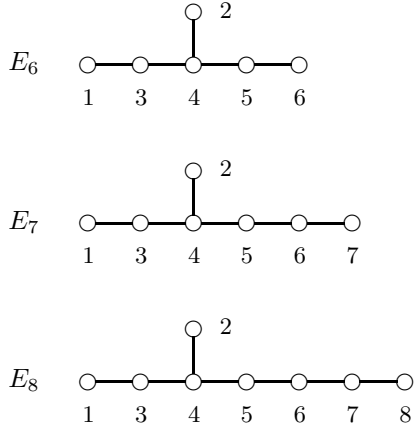

\section{Proofs}

\subsection{Formulas for characters of Kirillov-Reshetikhin modules}

Although there is no general explicit formula for the coefficients $a_k^{(i)}(\lambda)$ of (\ref{krm}) in type $E$, 
for some particular vertices $i\in I$, closed formulas are known. 
We shall use the following formulas of Chari \cite[Section 3]{Cha}.  
For convenience, we use the shorthand notation $\chi(\lambda)$ instead of $\chi(V(\lambda))$.

\medskip

Type $E_6$: 
\begin{eqnarray}\label{chari-E6}
&&\chi(W_k^{(i)}) = \chi (k\varpi_i) \quad \mbox{for } i=1,6, \\
&&\chi(W_k^{(2)}) = \displaystyle \sum_{r=0}^k \chi( r\varpi_2).\label{chari-E6-2}
\end{eqnarray}

Type $E_7$:
\begin{eqnarray}\label{chari-E7}
&&\chi(W_k^{(1)})= \displaystyle \sum_{r=0}^k \chi(r\varpi_1),\\
&&\chi(W_k^{(2)})= \displaystyle \sum_{r=0}^k \chi( r \varpi_2+(k-r)\varpi_7),\\
&&\chi(W_k^{(7)})= \chi( k\varpi_7).
\end{eqnarray}

Type $E_8$:
\begin{eqnarray}\label{chari-E8}
&&\chi(W_k^{(1)})= \displaystyle \sum_{r+s=0}^k \chi(r\varpi_1 + s\varpi_8),\\
&&\chi(W_k^{(8)})= \displaystyle \sum_{r=0}^k \chi( r \varpi_8).\label{chari-E88}
\end{eqnarray}

\subsection{Affine Weyl group action}

Let $W$ be the Weyl group of $\gg$ and let $s_i \:(i\in I)$ be its Coxeter generators. The action of $W$ on the weight lattice $P$ is given by
\begin{equation}
s_i\lambda = \lambda - (\lambda \mid\alpha_i)\alpha_i.
\end{equation} 

Let $\hat{W}$ be the corresponding affine Weyl group, and let 
 $s_0,s_i \:(i\in I)$ be its Coxeter generators. For every $l\in\mathbb{Z}^*$, we can extend the action of $W$ on $P$ into a \emph{level $l$} action of $\hat{W}$, defined as follows.
 The level $l$ action of $s_0$ on $P$ is given by \begin{equation}
 s_0\lambda = s_\theta \lambda + l\theta,\end{equation} where $\theta$ is the highest root and \begin{equation}
 s_\theta \lambda = \lambda - (\lambda \mid\theta)\theta.\end{equation}  The level $l$ action of each $s_i\: (i\in I)$ is the natural action of $W$. From now on, the level $l$ will be fixed as in (\ref{not}).

We can shift the level $l$ action of $\hat{W}$ on $P$ by  $\rho$. This creates the \emph{dot action} of $\hat{W}$ on $P$, defined by \begin{equation}
 w\cdot \lambda = w(\lambda+\rho)-\rho,
\end{equation} for all $w\in \hat{W}$ and $\lambda \in P$. 

The following useful result has been used by Lee \cite{Lee}\cite{Lee2} in his proof of the KNS conjecture for classical types; see also \cite[Section 16.3]{FMS}.

\begin{proposition}\label{aweyl}
Let $\lambda\in P_+$ and $w\in\hat{W}$ such that $w\cdot \lambda \in P_+$.
Let $l$ and $\zeta$ be as in (\ref{not}). Then we have
\[
d_\zeta(\lambda)= (-1)^{\ell(w)} d_\zeta (w\cdot \lambda),
\]
where $\ell(w)$ is the length of $w$ in $\hat{W}$.
\end{proposition}

Let us denote by $C_\ell^0$ the fundamental alcove of $P$ for the dot action of $\hat{W}$
of level $l=\ell+h$, that is, 
$$C_\ell^0 := \left\{ \lambda = \displaystyle \sum_{i\in I} \lambda_i \varpi_i\in P : 
\lambda_i \geq 0, \: \displaystyle \sum_{i\in I} a_i \lambda_i \leq \ell\right\}.$$
We cite the following result from Lee's thesis \cite{Leethesis}. 

\begin{proposition}[{\cite[Theorem 6.5.2]{Leethesis}}]\label{zeta+}
We have $d_\zeta(\lambda)>0$ for all $\lambda \in C_\ell^0$.
\end{proposition}

\begin{corollary}\label{Qk+}
We have $d_\zeta(W^{(i)}_k) >0$ for $0\leq k\leq \ds\frac{\ell}{a_i}$ and for all $i\in I$.
\end{corollary}

\begin{proof}
Recall the usual partial order on weights: 
\[
\la \ge \mu\quad \Longleftrightarrow \quad 
\la - \mu \in \bigoplus_{i\in I} \mathbb{N}\alpha_i.
\]
Let $\lambda= \sum_{i\in I} \lambda_i \varpi_i\in C_\ell^0$, 
and let $\mu=\lambda-\alpha_j= \sum_{i\in I} \mu_i \varpi_i$ for some $j\in I$.
Let $(c_{i,j})_{i,j\in I}$ be the Cartan matrix and $\hat{C}=(c_{i,j})_{i,j\in \{0\} \cup I}$ 
the affine Cartan matrix. 
We have $\alpha_j=\displaystyle \sum_{i\in I} c_{i j}\varpi_i$. Therefore, we get
\begin{equation}
\displaystyle \sum_{i\in I} \mu_i a_i = \displaystyle \sum_{i\in I} \lambda_i a_i - \displaystyle \sum_{i\in I} c_{ij} a_i.
\end{equation}
Now, the following identity is well known:
\begin{equation}
c_{0 j} + \displaystyle \sum_{i=1}^n a_i c_{ij}=0.
\end{equation}
Moreover, since $j\neq 0$, the coefficient $c_{0 j}$ is not on the diagonal of $\hat{C}$, hence $c_{0 j}\leq 0$. 
Therefore, we get
\begin{equation}
\displaystyle \sum_{i\in I} \mu_i a_i = \displaystyle \sum_{i\in I} \lambda_i a_i + c_{0j} 
\leq \displaystyle \sum_{i\in I} \lambda_i a_i \leq \ell.
\end{equation}
It follows that, if $\mu$ is dominant, $\mu\in C_\ell^0$. By induction, we get that if $\la\in C_\ell^0$ and
$\mu\le \la$ is dominant, then $\mu \in C_\ell^0$.

Finally, it is well-known that in the decomposition (\ref{decomp}), we have $a_k^{(i)}(\lambda)\neq0$ 
only if $\lambda \leq k\varpi_i$. 
We can then use Proposition \ref{zeta+} and the fact that $k\varpi_i\in C_\ell^0$ if $k\le \ell/a_i$
to finish our proof. \hfill $\Box$

\end{proof}

\subsection{Log-concave sequences}

\subsubsection{Definitions and basic properties}

A finite sequence of real numbers $(a_k)_{0\leq k \leq n}$ is called \emph{log-concave} if 
\begin{equation}
a_k^2 \geq a_{k-1} a_{k+1},\qquad (1\le k \le n-1).
\end{equation}
It is \emph{strictly} log-concave if $a_k^2 > a_{k-1} a_{k+1}$.
Equivalently, $(a_k)_{0\leq k \leq n}$ is log-concave if and only if 
\begin{equation}\label{logc} 
a_k a_l \geq a_{k+1} a_{l-1},\qquad (1\le k \le l \le n-1).
\end{equation} 
Strict log-concavity can also be characterized by equation (\ref{logc})  with a strict inequality.

Most of the following properties are well-known, see e.g. \cite{Sta},\cite{Wil}.
We include a short proof for the convenience of the reader.

\begin{proposition}\label{plogc}
\begin{enumerate}
\item [(i)]The product $(a_k b_k)_{0\leq k \leq n}$ of two positive log-concave sequences $(a_k)$ and $(b_k)$ is log-concave. Furthermore, if $(b_k)$ is strictly log-concave, then $(a_k b_k)$ is also strictly log-concave.
\item[(ii)] Let  $(a_k)_{0\leq k \leq n}$ be a strictly log-concave and positive sequence such that $a_{n-1}<a_n$. Define two new sequences  $(b_r)_{0\leq r \leq 2n}$ and $(c_r)_{0\leq r \leq 2n+1}$ by setting $b_r=a_r=b_{2n-r}$ and $c_r=a_r=c_{2n+1-r}$ for all $0\leq r \leq n$. Then the positive sequences $(b_r)_{0\leq r \leq 2n}$ and $(c_r)_{0\leq r \leq 2n+1}$ are strictly log-concave.
\item[(iii)]  If $(a_k)_{0\leq k \leq n}$ is strictly log-concave and positive, then the sum sequence $\displaystyle \left(\sum_{k=0}^d a_k\right)_{0\leq d \leq n}$ is strictly log-concave.

\end{enumerate}

\end{proposition}

\begin{proof}
Let $(a_k)$ and $(b_k)$ be two positive log-concave sequences. 
Then 

$$(a_k b_k)^2 \geq (a_{k-1}a_{k+1})(b_{k-1}b_{k+1})=(a_{k-1}b_{k-1})(a_{k+1}b_{k+1}),\quad (0<k<n),$$
hence $(a_k b_k)_{0\leq k\leq n}$ is log-concave. 
If moreover $b_k^2 > b_{k-1}b_{k+1}$, the above inequality is strict and $(a_k b_k)_{0\leq k\leq n}$ 
is strictly log-concave. This proves (i).

Let $(a_k)$ be a strictly log-concave and positive sequence such that $a_{n-1}<a_n$.  Then the sequence $(b_k)_{0\leq k \leq n}$ is strictly log-concave, and so is  $(b_k)_{n+1\leq k \leq 2n}$. Moreover, since $a_{n-1}<a_n$, we have
$$b_n^2 = a_n^2 > a_{n-1}^2 = b_{n-1}b_{n+1},$$
and $(b_n)_{0\leq k \leq 2n}$ is strictly log-concave.

Similarly the subsequences $(c_k)_{0\leq k \leq n}$ and $(c_k)_{n+1\leq k\leq 2n+1}$
are strictly log-concave. We have $c_n=c_{n+1}=a_n$. Moreover, since $a_{n-1}<a_n$, we have
$$c_n^2 =a_n^2 > a_{n-1}a_n = c_{n-1}c_{n+1} \quad \mbox{and}\quad c_{n+1}^2 =a_n^2 > a_{n-1}a_n = c_{n+2}c_n,$$
hence $(c_n)_{0\leq k \leq 2n+1}$ is strictly log-concave, which proves (ii).

Let $(a_k)_{0\leq k \leq n}$ be a strictly log-concave and positive sequence. We have
$$\begin{array}{ll}\left(\displaystyle \sum_{k=0}^{d-1} a_k \right) \left(\displaystyle \sum_{k=0}^{d+1} a_k\right) &= \displaystyle \left(\sum_{k=0}^d a_k\right)^2 -(a_d^2-a_{d+1}a_{d-1})-(a_d a_{d-1}-a_{d+1}a_{d-2})\\\\&-(a_d a_{d-2}-a_{d+1}a_{d-3}) \dots -(a_d a_1-a_{d+1}a_0)-a_d a_0.\end{array}$$
Since $a_d a_0>0$ and $a_d a_r-a_{d+1}a_{r-1}>0$ for all $1\leq r \leq d-1$ by (\ref{logc}), 
this proves (iii). \hfill$\Box$

\end{proof}

\subsubsection{Log-concave sequences and quantum dimensions}

Because of (\ref{not}) and (\ref{eq-zeta}), we know that $d_\zeta(\lambda)$ is a real number, and can be rewritten as
\begin{eqnarray}\label{sin}
d_\zeta(\lambda)= \displaystyle \prod_{\beta \in\Phi^+ : (\lambda|\beta)\neq 0} 
\frac{\sin\left(\ds\frac{\pi}{l}( (\lambda\mid\beta)+(\rho\mid\beta)  )\right)}{\sin\left(\ds\frac{\pi}{l}(\rho\mid\beta)\right)}.
\end{eqnarray}
Obviously, the sequence $d_\zeta(k\varpi_i)_{k\in\mathbb{Z}}$ is  $2l$-periodic.

\begin{proposition}\label{l1.8}

For every $i\in I$, the finite sequence $d_\zeta(k\varpi_i)$ $(0\leq k\leq {\ell}/{a_i})$ is strictly log-concave and positive.

\end{proposition}

\begin{proof}
The positivity follows immediately from Proposition \ref{zeta+}.
Define, for $\beta\in \Phi^+$,
\begin{equation}
a_{k,\beta}:=\sin\frac{\pi}{l}(k(\varpi_i\mid\beta)+(\rho\mid\beta)),\qquad (0\le k \le {\ell}/{a_i}).
\end{equation}
We have
\begin{equation}\label{lc}
a_{k,\beta}^2 - a_{k-1,\beta}a_{k+1,\beta} = \displaystyle \frac{1-\cos\frac{2\pi}{l}(\varpi_i\mid\beta)}{2}\ge 0,
\qquad (0< k < {\ell}/{a_i}).
\end{equation}
Therefore, $(a_{k,\beta })_{0\leq k \leq {\ell}/{a_i}}$ is log-concave. 
Moreover for $\beta = \theta$ the inequality (\ref{lc}) is strict, 
hence, by Proposition \ref{plogc}(i), the sequence
$(d_\zeta(k\varpi_i))_{0\leq k \leq {\ell}/{a_i}}$ is  strictly log-concave.
\hfill $\Box$

\end{proof}

\subsection{Miscellaneous results on root systems and $Q$-systems}

For every $i\in I$, we denote by $\delta_i$ the number of positive roots $\beta$ such that 
$(\varpi_i|\beta)$ is odd. 
By inspecting the list of positive roots, we see that $\delta_i$ is even for all $i\in I$ in type $E_6$. 
In type $E_7$, $\delta_2$, $\delta_5$, $\delta_7$ are odd, and $\delta_1$, $\delta_3$, $\delta_4$, $\delta_6$ are even.
In type $E_8$, $\delta_i$ is even for all $i\in I$.

The following useful lemma can also be checked by inspection of the list of positive roots in type $E$.
For a general proof, see \cite[Proposition 3.8]{Lee2}:
\begin{lemma}\label{lee}
Let $i\in I$ and $r\in \llbracket 1,h-1\rrbracket$. There exists $\beta\in\Phi^+$ such that
$(\varpi_i\mid \beta) = 1$ and $(\rho\mid \beta) = r$.
\end{lemma}

We also quote the following result of Lee, which will be used in type $E_6$ for proving property (iv) of Conjecture \ref{KNS-conj}.
\begin{theorem}[{\cite[Theorem 5.3.6]{Leethesis}}]\label{leeQ}
Let $\mathbf{Q}=(Q_k^{(i)})_{0\leq k \leq \ell}$ be the unique positive solution of a level $\ell$-restricted $Q$-system.  
Then $\mathbf{Q}$ satisfies the following properties:
\begin{enumerate}
\item [(i)]  (symmetry) $Q_{\ell-k}^{(i)}=Q_k^{(i)}$ for all $i \in I$ and $k\in\llbracket 0,\ell\rrbracket$.
\item [(ii)] (unimodality) $Q_k^{(i)} < Q_{k+1}^{(i)}$ for $k\in \llbracket 0,\lfloor \ell/2\rfloor-1\rrbracket$.

\end{enumerate}

\end{theorem}

For convenience, from now on, we set $Q_k^{(i)}=d_\zeta(W_k^{(i)})$ for every $k\in\mathbb{Z}$.

\subsection{Type $E_6$}\label{proofE6}

Let us first consider the extremal nodes 1, 2, 6.

\begin{proposition}\label{E6ext}
For $i=1,2,6$, in type $E_6$, the following properties hold:
\begin{enumerate}
\item [(i)] The sequences $Q_k^{(i)}$ are $l$-periodic.
\item [(ii)] We have $Q_{k}^{(i)}=0$ for $k\in\llbracket \ell+1,l-1\rrbracket$.
\item [(iii)] The sequences $Q_k^{(i)}$ are positive when $k\in\llbracket 0,\ell\rrbracket$. Moreover, the sequence $Q_k^{(2)}$ is strictly log-concave for $0\leq k \leq \ell$.
\item[(iv)] (symmetry property) We have $Q_{\ell-k}^{(i)}=Q_k^{(i)}$ for $k\in\llbracket 0,\ell\rrbracket$. 

\end{enumerate}

\end{proposition}

\begin{proof}

\textbf{(a) Nodes 1 and 6:} Formula~(\ref{chari-E6}) implies that for $i=1,6$, we have $$Q_k^{(i)} = d_\zeta (k\varpi_i).$$
We have $a_1=a_6=1$, so  
that Proposition \ref{l1.8} implies that $Q_k^{(1)}>0$ and $Q_k^{(6)}>0$ for all $0\leq k \leq \ell$. 
This proves (iii) at nodes 1 and 6.

Moreover, since $\delta_i$ is even, we can deduce from (\ref{sin}) that the sequence 
$(Q_k^{(i)})_{k\in\mathbb{Z}}$ is $l$-periodic. In particular, we have $Q_l^{(i)}=1$. This proves (i).

We have  $(\varpi_i\mid\beta)\in\{0,1\}$ for every $\beta\in\Phi^+$ for $i=1, 6$. 
When $k=l-r$, with $r\in\llbracket 1,h-1\rrbracket$, since $\delta_i$ is even, we get
$$Q_{l-r}^{(i)} = \displaystyle \prod_{\beta \in\Phi
^+:(\varpi_i|\beta)=1} \displaystyle \frac{\sin \frac{\pi}{l}((\rho|\beta)-r)}{\sin\frac{\pi}{l}(\rho|\beta)}.$$
It follows from Lemma \ref{lee} that $Q_{l-r}^{(i)}=0$ for $r\in\llbracket 1,h-1\rrbracket$. This proves (ii).

We have, for $k\in\llbracket 0,\ell\rrbracket$,
$$Q_{\ell-k}^{(i)}= \displaystyle \prod_{\beta \in\Phi^+:(\varpi_i|\beta)=1} \displaystyle \frac{\sin\frac{\pi}{l}( k+h-(\rho|\beta) ) }{\sin\frac{\pi}{l}(\rho|\beta)}.$$
By inspection of the list of positive roots, we observe that for $r\in\llbracket 1,h-1\rrbracket$, 
the sets 
\[
\{\beta \mid (\varpi_i|\beta)=1,\ (\rho|\beta  )=r \} \quad\mbox{and}\quad
\{\beta\mid (\varpi_i|\beta)=1,\: h-(\rho|\beta)=r\}
\]
have the same number of elements. Therefore, we can permute the corresponding factors in $Q_{\ell-k}^{(i)}$, which gives $Q_{\ell-k}^{(i)}=Q_k^{(i)}$. This proves (iv).

\medskip
\textbf{(b) Node 2:} Formula~(\ref{chari-E6-2}) implies that
\begin{equation}\label{E62} Q_k^{(2)}= \displaystyle \sum_{r=0}^k d_\zeta(r\varpi_2).\end{equation}
Since $a_2=2$, Proposition \ref{plogc}(iii) and
Proposition \ref{l1.8} imply that 
the sequence $(Q_k^{(2)})_{0\leq k \leq \lfloor \frac{\ell}{2}\rfloor}$ is positive, strictly increasing, 
and strictly log-concave. 
We will use the following lemma to prove (iv) at node 2.
\begin{lemma}\label{zetaE6}
For $k\in\llbracket 0,\ell+1\rrbracket$, we have
\begin{equation}
d_\zeta(k\varpi_2)=-d_\zeta((\ell+1-k)\varpi_2).
\end{equation}
\end{lemma}
\begin{proof}[Proof of Lemma \ref{zetaE6}]
In type $E_6$, we have $\theta=\varpi_2$.  The level $l$ dot action of $s_0$ on $k\varpi_2$ gives
$$s_0 \cdot (k\varpi_2)=k\varpi_2+\rho+(l-h+1-2k)\varpi_2-\rho = (\ell+1-k)\varpi_2$$
and the result follows by Proposition \ref{aweyl}. \hfill $\Box$

Let us assume that $k\in\llbracket 0,\lfloor\frac{\ell}{2}\rfloor\rrbracket$. By Lemma \ref{zetaE6},
$$Q_{\ell-k}^{(2)}= \displaystyle \sum_{r=0}^{\ell-k} d_\zeta(r\varpi_2) =\displaystyle \sum_{r=0}^{\lfloor \frac{\ell}{2}\rfloor} d_\zeta (r\varpi_2) - \displaystyle \sum_{s=k+1}^{\ell-\lfloor \frac{\ell}{2}\rfloor} d_\zeta(s\varpi_2).$$
If $\ell$ is odd, this difference is equal to 
$\displaystyle \sum_{r=0}^k d_\zeta(r\varpi_2)-d_\zeta((\ell+1)\varpi_2/2)$; if $\ell$ is even
it is equal to $\displaystyle \sum_{r=0}^k d_\zeta(r\varpi_2)$. 
Since for $\ell$ odd, $d_\zeta((\ell+1)\varpi_2/2)=0$ by Lemma \ref{zetaE6}, we get in both cases $Q_{\ell-k}^{(2)}=Q_k^{(2)}$, 
which proves (iv).

This also implies the positivity property, since we have seen that $Q_k^{(2)}>0$ for $0\leq k \leq \lfloor\frac{\ell}{2}\rfloor.$
By Proposition \ref{plogc}(ii), we can conclude that the sequence $(Q_k^{(2)})_{0\leq k\leq \ell}$ is strictly log-concave, which proves (iii).

Lemma \ref{zetaE6} implies that $Q_{\ell+1}^{(2)}=0.$ 

By (\ref{E62}), we have
$$Q_{\ell+k+1}^{(2)}=Q_{\ell+k}^{(2)}+d_\zeta((k+1)\varpi_2).$$
Lemma \ref{lee} implies that $d_\zeta((k+1)\varpi_2)=0$ for $k\in\llbracket \ell+1,\ell+h\rrbracket$. Therefore, we conclude by induction on $k$ that $Q_{\ell+k}^{(2)}=0$ for $k\in\llbracket 1,h-1\rrbracket$. This proves (ii).

Property (ii) implies that $Q_l^{(2)}=d_\zeta(l\varpi_2)=(-1)^{\delta_2}=1$. 
Since the sequence $(d_\zeta(k\varpi_2))_{k\in\mathbb{Z}}$ is $l$-periodic, 
it follows that $Q_{l+k}^{(2)}=Q_k^{(2)}$ for all $k\in\Z$, which proves (i). 
This finishes the proof of Proposition \ref{E6ext}. \hfill$\Box$

\end{proof}
\end{proof}
The remaining nodes are numbered 3, 4, 5. Node 3 (resp. 4, 5) is the only neighbour of node 1 (resp. 2, 6) in the Dynkin diagram $E_6$ (see Figure 1).

\begin{lemma}\label{E6nei}
For $j=3,4,5$, the following properties hold:
\begin{enumerate}
\item [(i)] The sequences $Q_k^{(j)}$ are $l$-periodic.
\item [(ii)] We have $Q_{k}^{(j)}=0$ for $k\in\llbracket \ell+1,l-1\rrbracket$.
\item [(iii)] The sequences $Q_k^{(j)}$ are  positive when $k\in\llbracket 0,\ell\rrbracket$.
\item[(iv)] (symmetry property) We have $Q_{\ell-k}^{(j)}=Q_k^{(j)}$ for $k\in\llbracket 0,\ell\rrbracket$. 

\end{enumerate}
\end{lemma}

\begin{proof}
The same method applies at nodes 3, 4, and 5. At a fixed node $i=1,2,6$ and $j$ its unique neighbour, the $Q$-system reads:
\begin{equation}\label{e6n}
(Q_k^{(i)})^2=Q_{k-1}^{(i)} Q_{k+1}^{(i)} + Q_k^{(j)}.\end{equation}
Properties (i), (ii), and (iv) follow easily from equation (\ref{e6n}) and Proposition~\ref{E6ext}.
Since $a_3=a_5=2$, Corollary \ref{Qk+} gives $Q_k^{(3)}>0$ and $Q_k^{(5)}>0$ for $0\leq k \leq \frac{\ell}{2}$. Moreover, the symmetry property implies that $Q_{\ell-k}^{(5)} = Q_k^{(5)}$ and $Q_{\ell-k}^{(3)}=Q_k^{(3)}$ for $0\leq k \leq \ell$, which proves (iii) at nodes 3 and 5.

By Proposition \ref{E6ext}(iii), the sequence $(Q_k^{(2)})_{0\leq k \leq \ell}$ is strictly log-concave, thus we have
$$Q_k^{(4)}=(Q_k^{(2)})^2-Q_{k-1}^{(2)}Q_{k+1}^{(2)} >0$$
for $k\in\llbracket 1,\ell -1\rrbracket$. The symmetry property implies that $Q_\ell^{(4)}=Q_0^{(4)}=1>0$, which finishes the proof. \hfill $\Box$

\end{proof}

We then use Theorem \ref{leeQ} to obtain property (iv) 
of Conjecture \ref{KNS-conj} in type $E_6$, and thus complete the proof of Theorem \ref{knsE678} in this case.

\subsection{Type $E_7$}\label{proofE7}

In this section, we shall frequently use the list of positive roots of $E_7$. 
For the convenience of the reader, this list is reproduced below in the table of Appendix~A.
In this table each root is numbered from $1$ to $63$, and we shall denote by $\beta_n$ the
the root with number $n$.

As in type $E_6$, we first consider the extremal nodes, numbered 1  and 7 (see Figure 1).

\begin{lemma}\label{e717}
For $i=1,7$ in type $E_7$, the following properties hold:
\begin{enumerate}
\item [(i)] The sequence $Q_k^{(1)}$ is $l$-periodic, and we have $Q_{l+k}^{(7)}=-Q_k^{(7)}$ for all $k\in\mathbb{Z}$.
\item [(ii)] We have $Q_{k}^{(i)}=0$ for $k\in\llbracket \ell+1,l-1\rrbracket$.
\item [(iii)] The sequences $(Q_k^{(i)})_{0\leq k \leq\ell}$ are positive. Moreover, the sequence $(Q_k^{(1)})_{0\leq k \leq\ell}$ is strictly log-concave.
\item[(iv)] (symmetry property) We have $Q_{\ell-k}^{(i)}=Q_k^{(i)}$ for $k\in\llbracket 0,\ell\rrbracket$. 

\end{enumerate}
\end{lemma}

\begin{proof}
\textbf{(a) Node 7:} By Formula (18), we have $Q_k^{(7)}=d_\zeta(k\varpi_7)$. Since $\delta_7=27$ and $a_7=1$, we have, for all $k$,
$$ Q_{l+k}^{(7)}=\displaystyle \prod_{\beta \in \Phi^+} \frac{\sin\frac{\pi}{l}\left(   l+k+ (\rho\mid\beta)\right)}{\sin\frac{\pi}{l}(\rho\mid\beta)}=(-1)^{\delta_7}\displaystyle \prod_{ (\varpi_7\mid\beta)=1} \frac{\sin\frac{\pi}{l}\left(k+ (\rho\mid\beta)\right)}{\sin\frac{\pi}{l}(\rho\mid\beta)}=-Q_k^{(7)},$$
which proves $(i)$. For properties $(ii),(iii),(iv)$, the proof is the same as the proof of Proposition \ref{E6ext} at nodes 1 and 6.

\textbf{(b) Node 1:} By Formula (16), we have $Q_k^{(1)}=\displaystyle \sum_{r=0}^k d_\zeta(r\varpi_1)$, 
where $\varpi_1=\theta$ is the longest root in type $E_7$. 
The proof is the same as the proof of Proposition \ref{E6ext} at node 2. \hfill$\Box$

\end{proof}

Node 3 (resp. 6) is the unique neighbour of node 1 (resp. 7).

\begin{lemma}\label{e7n}
For $j=3,6$, in type $E_7$, the following properties hold:
\begin{enumerate}
\item [(i)] The sequences $Q_k^{(j)}$ are $l$-periodic.
\item [(ii)] We have $Q_{k}^{(j)}=0$ for $k\in\llbracket \ell+1,l-1\rrbracket$.
\item [(iii)] The sequences $Q_k^{(j)}$ are  positive when $k\in\llbracket 0,\ell\rrbracket$.
\item[(iv)] (symmetry property) We have $Q_{\ell-k}^{(j)}=Q_k^{(j)}$ for $k\in\llbracket 0,\ell\rrbracket$. 

\end{enumerate}
\end{lemma}

\begin{proof}
The positivity of $Q^{(3)}_k$ follows from the log-concavity of $Q^{(1)}_k$.
Apart from that,
the proof is the same as that of Lemma \ref{E6nei}. \hfill$\Box$
\end{proof}

We can now prove two results at nodes 4 and 5.

\begin{lemma}\label{sym45}
For $j=4,5$ in type $E_7$, the following properties hold:
\begin{enumerate}
\item [(i)] We have $Q_{l+k}^{(5)}=-Q_k^{(5)}$ and $Q_{l+k}^{(4)}=Q_k^{(4)}$ for $k\in\llbracket 0,\ell\rrbracket$ mod $l$.
\item [(ii)] (symmetry property) We have $Q_{\ell-k}^{(j)}=Q_k^{(j)}$ for $k\in\llbracket 0,\ell\rrbracket$.

\end{enumerate}
\end{lemma}

\begin{proof}
We use the same method at both nodes. Thus, at node 5 (resp. node~4), we use the $Q$-system at node 6 (resp. node 3), which reads:
$$(Q_k^{(6)})^2 = Q_{k-1}^{(6)} Q_{k+1}^{(6)} + Q_k^{(7)} Q_k^{(5)}.$$
We deduce both properties from the equation above, since they are known at nodes 6 and 7, and $Q_k^{(7)} >0$ for $k\in\llbracket 0,\ell\rrbracket$.\hfill$\Box$

\end{proof}

\begin{lemma}\label{e7zero2}
We have $Q_{\ell+k}^{(2)}=0$ for $k \in \llbracket 1, h-1\rrbracket$.
\end{lemma}

\begin{proof}
Formula (17) gives
$$Q_k^{(2)} =  \displaystyle \sum_{r=0}^k d_\zeta(r\varpi_2+(k-r)\varpi_7).$$ 
Here $a_2=2$ and $a_7=1$.

Let $w= s_0 s_1 s_3 s_4 s_5 s_6 s_7 s_6 s_5 s_4 s_3 s_1 s_0 \in \widehat{W}$.
One can check that, for $p,q\in\Z$, 
$$w \cdot (p \varpi_2 + q \varpi_7)=(\ell-p+7)\varpi_2 + (2p+q-\ell-7)\varpi_7.$$
Therefore, since $l(w)=13$, we can use Proposition \ref{aweyl} to deduce that
$$d_\zeta(p \varpi_2 + q \varpi_7 )=- d_\zeta( (\ell-p+7)\varpi_2 + (2p+q-\ell-7)\varpi_7 ), $$
for all $0\leq p \leq \ell+7$ and $q \geq \ell+7-2p$.
In particular, for $0\leq k \leq l$, we have, for $0\leq r \leq\ell+7$ and $\ell+7-k \leq r \leq k$,
$$d_\zeta(r \varpi_2 + (k-r)\varpi_7)=-d_\zeta((\ell-r+7)\varpi_2+(k+r-\ell-7)\varpi_7).$$
When $k=\ell+7$, we get
$$d_\zeta(r\varpi_2+(\ell+7-r)\varpi_7)=-d_\zeta((\ell+7-r)\varpi_2+\varpi_7),$$
for $0\leq r \leq \ell+7$. This implies that $Q_{\ell+7}^{(2)}=0$.
When $k=\ell+6$, we have 
\[
d_\zeta(r\varpi_2+(\ell+6-r)\varpi_7)=-d_\zeta((\ell-r+7)\varpi_2+(r-1)\varpi_7),
\]
for $1\leq r \leq \ell+6$, so that
$$\begin{array}{ll}
  Q_{\ell+6}^{(2)}&= \displaystyle \sum_{r=0}^{\ell+6} d_\zeta(r\varpi_2+(\ell+6-r)\varpi_7)\\\\
&= \displaystyle \sum_{r=1}^{\ell+6} d_\zeta(r\varpi_2+(\ell+6-r)\varpi_7)+d_\zeta((\ell+6)\varpi_7)\\\\
&= 0+ Q_{\ell+6}^{(7)} =0,
\end{array}$$
since we already have $Q_{\ell+k}^{(7)}=0$ for $1\leq k \leq h-1$ by Lemma \ref{e717}. 
When $k=\ell+5$, we have
$$\begin{array}{ll}
Q_{\ell+5}^{(2)} &= \displaystyle \sum_{r=0}^{\ell+5} d_\zeta(r\varpi_2 + (\ell+5-r)\varpi_7) \\\\&= \displaystyle \underset{0}{\underbrace{\sum_{r=2}^{\ell+5} d_\zeta(r\varpi_2 + (\ell+5-r)\varpi_7)}} + \underset{=Q_{\ell+5}^{(7)}=0}{\underbrace{d_\zeta((\ell+5)\varpi_7)}} + d_\zeta(\varpi_2+(\ell+4)\varpi_7). \\\\
\end{array}$$
The positive root $\beta = \alpha_1+\alpha_2+ 2\alpha_3+3\alpha_4 + 3\alpha_5 + 2\alpha_6+\alpha_7$ 
(this is the root $\beta_{59}$ in the table of Appendix A) is such that $(\rho\mid\beta)=h-5$ and $(\varpi_2\mid\beta)=(\varpi_7\mid\beta)=1$, so that 
$$(\varpi_2\mid\beta)+(\ell+4)(\varpi_7\mid\beta) +(\rho\mid\beta)= 1+(l-h+4)\cdot 1+h-5 =l,$$
thus $\sin\frac{\pi}{l}((\varpi_2\mid\beta)+(\ell+4)(\varpi_7\mid\beta) +(\rho\mid\beta))=0$. 
Therefore the product $d_\zeta(\varpi_2+(\ell+4)\varpi_7)$ vanishes, and we have $Q_{\ell+5}^{(2)}=0$. 

The same method allows us to show that $Q_{\ell+k}^{(2)}=0$ for $1\leq k\leq 4$. We have
\[
Q_{\ell+k}^{(2)}= \displaystyle \underset{0}{\underbrace{\sum_{r=7-k}^{\ell+k} d_\zeta(r\varpi_2+(\ell+k-r)\varpi_7)}} 
+ \underset{0}{\underbrace{Q_{\ell+k}^{(7)}}}+\displaystyle\sum_{r=1}^{6-k} d_\zeta(r\varpi_2+(\ell+k-r)\varpi_7). 
\]
By inspecting the list of positive roots for $E_7$, 
we observe that each term of the second sum above vanishes: for $k=4$, we use roots $\beta_{58}$ and $\beta_{56}$; for $k=3,2$, roots $\beta_{56},\beta_{58}$ and $\beta_{60}$ each cancel a term; for $k=1$, we use roots $\beta_{62},\beta_{61},\beta_{60},\beta_{58}$ and $\beta_{56}$. 

When $k=\ell+8$, we have 
\[
\begin{array}{ll}Q_{\ell+8}^{(2)}&= \displaystyle \sum_{r=0}^{\ell+8} d_\zeta(r\varpi_2 + (\ell+8-r)\varpi_7) \\[5mm]
&= \displaystyle \sum_{r=0}^{\ell+7} d_\zeta(r\varpi_2 + (\ell+8-r)\varpi_7) + d_\zeta((\ell+8)\varpi_2). 
\end{array}
\]
The first sum above is zero, since 
\[
d_\zeta(r\varpi_2+(\ell+8-r)\varpi_7)=-d_\zeta((\ell+7-r)\varpi_2+(r+1)\varpi_7).
\]
The product $d_\zeta((\ell+8)\varpi_2)$ is zero, because the factor corresponding to $\beta_{52}$
 is zero. Therefore, we have $Q_{\ell+8}^{(2)}=0$.
The same method allows us to show that $Q_{\ell+k}^{(2)}=0$ for $9\leq k \leq h-1$. We have
$$\begin{array}{ll}
Q_{\ell+k}^{(2)}= \displaystyle \underset{0}{\underbrace{\sum_{r=0}^{\ell+7} d_\zeta(r\varpi_2 + (\ell+k-r)\varpi_7)}}+ \displaystyle \sum_{r=\ell+8}^{\ell+k} d_\zeta(r\varpi_2 + (\ell+k-r)\varpi_7).
\end{array}$$
By inspecting the list of positive roots for $E_7$, we observe that each term of the second sum 
above vanishes: for $k\in\llbracket 9,13\rrbracket$, we use respectively roots $\beta_{47},\beta_{44},\beta_{39},\beta_{36},\beta_{30}$; for $k\in\llbracket 14,17\rrbracket$, the following roots cancel terms: $\beta_{50},\beta_{46}$, $\beta_{43},\beta_{42},\beta_{32},\beta_{26},\beta_{20},\beta_{15},\beta_9,\beta_2$. This finishes our proof. \hfill $\Box$

\end{proof}

\begin{lemma}\label{e7one2}
We have $Q_\ell^{(2)}= 1$ and $Q_l^{(2)}=-1$.
\end{lemma}

\begin{proof}
We know from the proof of Lemma \ref{e7zero2} that
$$d_\zeta(r\varpi_2+(\ell-r)\varpi_7)=-d_\zeta((\ell+7-r)\varpi_2+(r-7)\varpi_7),$$
for $7 \leq r \leq \ell$. Therefore, we deduce from Formula (\ref{chari-E7}) that
\[
\begin{array}{ll}
Q_\ell^{(2)}&= \displaystyle \sum_{r=7}^{\ell-1} d_\zeta(r\varpi_2+(\ell-r)\varpi_7) 
+ \displaystyle \sum_{r=1}^6  d_\zeta(r\varpi_2+(\ell-r)\varpi_7) \\[5mm]
&\ \ \ +\ d_\zeta(\ell\varpi_7)\\
&= 0+ Q_\ell^{(7)} +  \displaystyle \sum_{r=1}^6  d_\zeta(r\varpi_2+(\ell-r)\varpi_7) .
\end{array}
\]
We already know that $Q_\ell^{(7)}=1$ from Lemma \ref{e717}. 
Again, by inspecting the list of positive roots for $E_7$, we observe that each term of the last sum 
vanishes, thanks to roots $\beta_{63},\beta_{62},\beta_{61},\beta_{60},\beta_{58},$ and $\beta_{56}$. Therefore, we have $Q_\ell^{(2)}=1$.

We have $l=\ell+h=\ell+18$, so that 
$$d_\zeta(\varpi_2+(\ell+18-r)\varpi_7)=-d_\zeta((\ell+7-r)\varpi_2+(r+11)\varpi_7).$$
This gives
$$\begin{array}{ll} Q_l^{(2)}= Q_{\ell+18}^{(2)} 
&= \displaystyle \underset{0}{\underbrace{\sum_{r=0}^{\ell+7} d_\zeta(r\varpi_2+(\ell+18-r)\varpi_7)}} 
+ d_\zeta((\ell+18)\varpi_2)\\&+ \displaystyle \sum_{r=\ell+8}^{\ell+17} d_\zeta(r\varpi_2+(\ell+18-r)\varpi_7).\end{array}$$
Again, by inspecting the list of positive roots for $E_7$, we observe that each term of the last sum above 
vanishes, thanks to roots  $\beta_{50},\beta_{46}$, $\beta_{43},\beta_{42},\beta_{32},\beta_{26},\beta_{20},\beta_{15},$ $\beta_9,$ and $\beta_2$. For example, the positive root $\alpha_2$ cancels out the product $d_\zeta((\ell+17)\varpi_2+\varpi_7)$, 
since we have $(l-h+17)\cdot 1 + 0 + 1= l$. 
Therefore, we have $Q_l^{(2)}=d_\zeta(l\varpi_2)=(-1)^{\delta_2}=-1$ since $\delta_2$ is odd. \hfill $\Box$

\end{proof}

\begin{lemma}
We have $Q_k^{(2)}>0$ for $ k\in\llbracket 0,  \lfloor \frac{\ell}{2}\rfloor\rrbracket$ and $Q_{\ell-k}^{(2)}=Q_k^{(2)}$ for $k \in\llbracket 0, \ell\rrbracket$.

\end{lemma}

\begin{proof}
We have $a_2=2$, so it follows from Proposition~\ref{Qk+} that 
$Q_k^{(2)}>0$ for $ k\in\llbracket 0,  \lfloor \frac{\ell}{2}\rfloor\rrbracket$.
Let us consider $Q_{\ell-1}^{(2)}$. We have
$$\begin{array}{ll}Q_{\ell-1}^{(2)} &= \displaystyle \underset{0}{\underbrace{\sum_{r=8}^\ell d_\zeta(r\varpi_2
+(\ell-r-1)\varpi_7)}}+d_\zeta((\ell-1)\varpi_7)+\displaystyle   d_\zeta(\varpi_2+(\ell-2)\varpi_7) \\
&+ \ds\sum_{r=2}^7 d_\zeta(r\varpi_2+(\ell-r-1)\varpi_7).\end{array}$$
We already know that $d_\zeta(\varpi_7)=d_\zeta((\ell-1)\varpi_7)$.
Again, by inspecting the list of positive roots for $E_7$, we observe that each term of the last sum above 
vanishes, thanks to roots $\beta_{63},\beta_{62},\beta_{61},\beta_{60},\beta_{58},$ and $\beta_{56}$. 

Therefore, we need to show that $d_\zeta(\varpi_2+(\ell-2)\varpi_7)=d_\zeta(\varpi_2).$
We have
$$
\begin{array}{ll}d_\zeta(\varpi_2+(\ell-2)\varpi_7)
=\displaystyle\prod_{\beta\in\Phi^+}  
\ds \frac{\sin\frac{\pi}{l}((l-20)(\varpi_7\mid\beta)+(\varpi_2\mid\beta)+(\rho\mid\beta))}{\sin\frac{\pi}{l}(\rho\mid\beta)}\\\\
=\ds\prod_{\beta:(\varpi_7\mid\beta)=0} \frac{\sin\frac{\pi}{l}((\varpi_2\mid\beta)+(\rho\mid\beta))}{\sin\frac{\pi}{l}(\rho\mid\beta)} 
\prod_{\beta:(\varpi_7\mid\beta)=1}  \ds  \frac{\sin\frac{\pi}{l}(l-20+(\varpi_2\mid\beta)+(\rho\mid\beta))}{\sin\frac{\pi}{l}(\rho\mid\beta)}\\ \\
= -\ds\displaystyle\prod_{\beta:(\varpi_7\mid\beta)=0} \frac{\sin\frac{\pi}{l}((\varpi_2\mid\beta)+(\rho\mid\beta))}{\sin\frac{\pi}{l}(\rho\mid\beta)} 
\prod_{\beta:(\varpi_7\mid\beta)=1}  \displaystyle \frac{\sin\frac{\pi}{l}(-20+(\varpi_2\mid\beta)+(\rho\mid\beta))}{\sin\frac{\pi}{l}(\rho\mid\beta)}.    \end{array}              $$
For $(\varpi_7\mid\beta)=1$, here are the values of $(\varpi_2\mid\beta)+(\rho\mid\beta)$, when $\beta\in\Phi^+$:
$$\underset{(\varpi_2\mid\beta)=2}{\underbrace{19,18,17,16,15,14}},\underset{(\varpi_2\mid\beta)=1}{\underbrace{14,13,12,12,11,11,10,10,10,9,9,8,8,7,6}},\underset{(\varpi_2\mid\beta)=0}{\underbrace{5,6,4,3,2,1}}.$$
Therefore, replacing $(\varpi_2\mid\beta)+(\rho\mid\beta)$ by $-20+(\varpi_2\mid\beta)+(\rho\mid\beta)$ yields the opposite of each value above. Since $\delta_7=27$, the second product above is such that
$$\begin{array}{l}
d_\zeta(\varpi_2+(\ell-2)\varpi_7)=\displaystyle\prod_{\beta\in\Phi^+}  \displaystyle \frac{\sin\frac{\pi}{l}((l-20)(\varpi_7\mid\beta)+(\varpi_2\mid\beta)+(\rho\mid\beta))}{\sin\frac{\pi}{l}(\rho\mid\beta)}\\\\=\displaystyle  -\displaystyle\displaystyle\prod_{\beta:(\varpi_7\mid\beta)=0} \frac{\sin\frac{\pi}{l}((\varpi_2\mid\beta)+(\rho\mid\beta))}{\sin\frac{\pi}{l}(\rho\mid\beta)}  \cdot \left( -\prod_{\beta:(\varpi_7\mid\beta)=1}  \displaystyle \frac{\sin\frac{\pi}{l}((\varpi_2\mid\beta)+(\rho\mid\beta))}{\sin\frac{\pi}{l}(\rho\mid\beta)}\right)\\\\
= d_\zeta(\varpi_2).
\end{array}$$
We can then conclude that $Q_{\ell-1}^{(2)}=Q_1^{(2)}>0$. 
Let us now proceed by induction on $k\ge 1$ by using the $Q$-system at node 2 for each $k\in\llbracket 1,\ell \rrbracket$, and Lemma \ref{sym45}(ii). Assuming that $Q_{\ell-k+1}^{(2)}=Q_{k-1}^{(2)}$ and $Q_{\ell-k}^{(2)}=Q_k^{(2)}$, we have:
$$\begin{array}{ll}
Q_{\ell-k}^{(4)}&= (Q_{\ell-k}^{(2)})^2 - Q_{\ell-k-1}^{(2)} Q_{\ell-k+1}^{(2)}\\\\
=Q_k^{(4)} &= (Q_k^{(2)})^2 - Q_{\ell-k-1}^{(2)} Q_{k-1}^{(2)} \\\\
&=  (Q_k^{(2)})^2 - Q_{k+1}^{(2)} Q_{k-1}^{(2)}.
\end{array}$$
Since by induction $Q_{k-1}^{(2)} >0$, we get $Q_{\ell-k-1}^{(2)}=Q_{k+1}^{(2)},$  and $Q_{\ell-k-1}^{(2)}>0.$

We can finally conclude that $Q_{\ell-k}^{(2)}=Q_k^{(2)}$ for $k\in\llbracket 0,\ell\rrbracket$. Furthermore, we have $Q_k^{(2)}>0$ for $k\in\llbracket 0,\ell\rrbracket$. \hfill $\Box$

\end{proof}

\begin{lemma}\label{e745zero}
For $i=4,5$, we have $Q_{\ell+k}^{(i)}=0$ for all $k\in\llbracket 1,h-1\rrbracket$. 
Moreover, the sequence $Q_k^{(4)}$ is $l$-periodic and $Q_{l+k}^{(5)}=-Q_k^{(5)}$ for all $k$.

\end{lemma}

\begin{proof}
\textbf{(a) Node 5:} According to Kleber \cite{Kle}, we have
\begin{equation}\label{kleb5} Q_1^{(5)} = d_\zeta(\varpi_5)+d_\zeta(\varpi_1+\varpi_7)+2d_\zeta(\varpi_2)+2d_\zeta(\varpi_7).\end{equation}
We have $a_1=a_2=2$, $a_7=1$ and $a_5=3$. If $\ell\geq 3$, the dominant weights occurring in (\ref{kleb5}) are all in $C_\ell^0$. Therefore, Proposition \ref{zeta+} implies that $Q_1^{(5)}$ is a sum of positive terms for $\ell \geq 3$. For $\ell=1, 2$, one can check by direct computation that $Q_1^{(5)}>0$.
The $Q$-system at node 5 reads
\begin{equation}\label{e75zero}(Q_{\ell+k}^{(5)})^2= Q_{\ell+k-1}^{(5)} Q_{\ell+k+1}^{(5)} +0,\end{equation}
for all $k\in\llbracket 1,h-1\rrbracket$, because $Q_{\ell+k}^{(6)}=0$. Moreover, we have
$$ \begin{array}{ll} (Q_\ell^{(5)})^2&=1=Q_{\ell+1}^{(5)} Q_{\ell-1}^{(5)} + 1\\\\&= Q_{\ell+1}^{(5)} \underset{\neq 0}{\underbrace{Q_1^{(5)}}} +1,
\end{array}$$
which implies that $Q_{\ell+1}^{(5)}=0$. We can then use (\ref{e75zero}) to show by induction on $k$ that $Q_{\ell+k}^{(5)}=0$ for all $k\in\llbracket 1,h-1\rrbracket$. We use again the $Q$-system to deduce that $Q_{l+k}^{(5)}=-Q_k^{(5)}$ for all $k$.

\medskip
\textbf{(b) Node 4:} According to Kleber \cite{Kle}, we have
$$\begin{array}{ll}Q_1^{(4)} &= 2+ 4d_\zeta(\varpi_1)+d_\zeta(2\varpi_1)+3d_\zeta(\varpi_3)+d_\zeta(\varpi_4)+4d_\zeta(\varpi_6)+d_\zeta(2\varpi_7) \\&\quad+ d_\zeta(\varpi_1+\varpi_6)+2d_\zeta(\varpi_2+\varpi_7).
\end{array}$$
The same method as for node 5 gives $Q_1^{(4)}>0$ for $\ell\geq 4$, by Proposition \ref{zeta+}. For $\ell=1,2,3$, one can check by direct computation that $Q_1^{(4)}>0.$ We then use the $Q$-system at node 4  to conclude that $Q_{\ell+k}^{(4)}=0$ for all $k\in\llbracket 1,h-1\rrbracket$, and deduce the full $l$-periodicity at node 4. \hfill $\Box$

\end{proof}

Finally, we can only prove the following partial positivity result at nodes 4 and 5.
\begin{lemma}
We have 
\[
\begin{array}{lr}
Q_k^{(4)}>0,& (k\in \llbracket 0,\frac{\ell}{4}\rrbracket\cup\llbracket\frac{3\ell}{4},\ell\rrbracket),\\[5mm] 
Q_k^{(5)}>0,& (k\in\llbracket 0,\frac{\ell}{3}\rrbracket\cup\llbracket\frac{2\ell}{3},\ell\rrbracket).
\end{array}
\]
\end{lemma}

\begin{proof}
Since $a_4=4$, Corollary \ref{Qk+} implies that $Q_k^{(4)}>0$ for $0\leq k\leq \ell/4$.
Similarly, since $a_5 = 3$, we have $Q_k^{(5)}>0$ for $0\leq k \leq \ell/3$. 
We then use Lemma \ref{sym45}(ii) to conclude. \hfill $\Box$

\end{proof}

This completes the proof of Theorem \ref{knsE678} in type $E_7$.

\begin{remark}\label{remark-inf-logc}
{\rm
Following \cite{Br}, define
$\L$ to be the operator mapping the sequence of real numbers 
$(a_k)_{0\le k\le n}$ to the sequence $(b_k)_{0\le k\le n}$ defined by
\[
 b_k := a_k^2 - a_{k+1}a_{k-1},\qquad (0\le k\le n),
\]
where we have set $a_{-1}=a_{n+1} = 0$.
Hence $(a_k)$ is log-concave if $\L(a_k)$ is non-negative.
We say that $(a_k)$ is \emph{$i$-fold log-concave} if the $i$th iterate
$\L^i(a_k)$ is non-negative, and \emph{infinitely log-concave} if
it is $i$-fold log-concave for all $i\in \N$.
It is easy to deduce from the $Q$-system that the non-negativity of 
$Q_k^{(i)}$ for $0\le k\le \ell$ and $i=4,5$ amounts to the fact
that the sequence 
\begin{equation}\label{seq-above}
Q_k^{(7)} = d_\zeta(k\varpi_7),\qquad (0\le k \le \ell)
\end{equation}
is $3$-fold and $2$-fold log-concave. In fact we conjecture that
this sequence is infinitely log-concave. 

It was proved by Br\"and\'en \cite{Br} that a sufficient condition
for $(a_k)_{0\le k\le n}$ to be infinitely log-concave is that the
polynomial
\[
 P(x) = \sum_{k=0}^n a_k x^k
\]
has only real and negative roots. We have checked numerically that this
condition is satisfied for the above sequence (\ref{seq-above}) for 
$\ell \le 11$, but unfortunately it fails for $\ell =12$.

We believe that this log-concavity problem could be of independent interest
(see for example \cite{MNS,Br} for similar problems arising in combinatorics).
}
\end{remark}

\subsection{Type $E_8$}\label{proofE8}
In this section, we shall frequently use the list of positive roots of $E_8$. 
For the convenience of the reader, this list is reproduced below in the table of Appendix~B.
In this table each root is numbered from $1$ to $120$, and we shall denote by $\beta_n$ the
the root with number $n$.

We first consider the extremal node number 8 (see Figure 1).

\begin{lemma}\label{e88}
At node 8 in type $E_8$, the following properties hold:
\begin{enumerate}[(i)]
\item The sequence $Q_k^{(8)}$ is $l$-periodic.
\item (symmetry) We have $Q_{\ell-k}^{(8)}= Q_k^{(8)}$ for all $k\in\llbracket 0,\ell\rrbracket$. Moreover, $Q_\ell^{(8)}=1$.
\item (zeroes) We have $Q_{k}^{(8)}=0$ for all $k\in\llbracket \ell+1,l-1\rrbracket$.
\item (positivity) We have $Q_k^{(8)}>0$ for all $k\in\llbracket 0,\ell\rrbracket$.
\item The sequence $Q_k^{(8)}, 0\leq k \leq \ell$, is strictly log-concave.
\end{enumerate}
\end{lemma}

\begin{proof}
By Formula (\ref{chari-E88}), we have $Q_k^{(8)}=\displaystyle \sum_{r=0}^k d_\zeta(r\varpi_8)$, 
where $\varpi_8=\theta$ is the longest root in type $E_8$. 
The proof is the same as the proof of Proposition \ref{E6ext} at node 2.\hfill $\Box$ 
\end{proof}

Let us now consider the extremal node 1.

\begin{lemma}\label{e81}
Properties $(i)$ to $(iv)$ of Lemma \ref{e88} also hold at node 1.
Moreover, we have
$
Q_k^{(1)} < Q_{k+1}^{(1)}$ for $k\in \llbracket 0,\lfloor \ell/2\rfloor-1\rrbracket. 
$

\end{lemma}

\begin{proof}
By Formula (\ref{chari-E8}), we have
\begin{equation}\label{qk1}
Q_k^{(1)} = \displaystyle \sum_{r+s=0}^k d_\zeta(s\varpi_1+r\varpi_8)=Q_{k-1}^{(1)}+T_k,
\end{equation}
where $T_k = \displaystyle \sum_{r+s=k} d_\zeta(s\varpi_1+r\varpi_8)$.
Since $a_1=2$, we have $Q_k^{(1)}>0$ and $T_k >0$ for $0\leq k \leq \frac{\ell}{2}$. Therefore, the sequence $(Q_k^{(1)})_{0\leq k \leq \frac{\ell}{2}}$ is strictly increasing 
and positive. We first establish five useful formulas, numbered (39), (40), (41), 
 and (42).

\textbf{(a)} At level $l=\ell+30$, one has:
$$s_0 \cdot (s\varpi_1+r\varpi_8)= s\varpi_1+(\ell+1-2s-r)\varpi_8.$$
Thus for $0\leq r+2s \leq \ell+1$, Proposition \ref{aweyl} gives
\begin{equation}\label{dz1}
d_\zeta(s\varpi_1+r\varpi_8) = - d_\zeta(s\varpi_1+(\ell+1-2s-r)\varpi_8).
\end{equation}

\textbf{(b)} Let us study the terms  $d_\zeta(s\varpi_1+r\varpi_8)$ such that 
 $\lfloor\frac{\ell+13}{2}\rfloor \leq s\leq \ell+13$ and $r\geq 0$.

Let $\beta=\beta_{97}=2\alpha_1+2\alpha_2+3\alpha_3+4\alpha_4+3\alpha_5+2\alpha_6+\alpha_7= \varpi_1-\varpi_8 \in \Phi^+$. We have $(\varpi_1\mid\beta)=2$, $(\varpi_8\mid\beta)=0$ and $(\rho\mid\beta)=17$.

Define $\sigma_\beta = t_{(\ell+30)\beta} \circ s_\beta$, with $s_\beta : \lambda \mapsto \lambda - (\lambda\mid\beta)\beta \in \widehat{W}$ having signature \nolinebreak$-1$,  and $t_{(\ell+30)\beta}: \lambda \mapsto \lambda + (\ell+30)\beta \in \widehat{W}$ having signature $1$. Then $\sigma_\beta \in \widehat{W}$ has signature $-1$.

The dot action gives:
$$\sigma_\beta \cdot ( s\varpi_1+r\varpi_8 ) =(\ell+13-s)\varpi_1+ (2s+r-\ell-13)\varpi_8 .$$
Thus for $r\geq 0$ and $\lfloor\frac{\ell+13}{2}\rfloor \leq s\leq \ell+13$, Proposition \ref{aweyl} gives
\begin{equation}\label{dz97} d_\zeta (s\varpi_1+r\varpi_8  ) =-d_\zeta(  (\ell+13-s)\varpi_1+(2s+r-\ell-13)\varpi_8 ).\end{equation}

\textbf{(c)} Now define
$$\displaystyle D_1(m,r)=d_\zeta( (\lceil\frac{\ell }{2}\rceil+m-r)\varpi_1+2r\varpi_8 ),
$$
for $r\in\llbracket 0,\lceil\frac{\ell}{2}\rceil+m\rrbracket$ and $m-(\ell \mbox{ mod } 2)\in\llbracket 0,4\rrbracket$, and
$$ D_2(m,r)= d_\zeta( (\lceil\frac{\ell}{2}\rceil+m-r)\varpi_1+(2r+1)\varpi_8 ),
$$
for $r\in\llbracket 0,\lceil\frac{\ell}{2}\rceil+m\rrbracket$ and $m-(\ell \mbox{ mod } 2)\in\llbracket 0,5 \rrbracket$. 
We claim that, for these values of $r$ and $m$, there holds
\begin{equation}\label{dzD1}
D_1(m,r)=0, \quad D_2(m,r)=0.
\end{equation} Indeed, the roots $\beta_{120},\beta_{118},\beta_{116},\beta_{114}, \beta_{109}$, $\beta_{104}$, and $\beta_{57}$ make $D_1(m,r)$ vanish if $\ell$ is odd, and they cancel $D_2(m,r)$ if $\ell$ is even. On the other hand, the roots $\beta_{119}$, $\beta_{117}$, $\beta_{115}$, $\beta_{111}$, $\beta_{107}$, and $\beta_{101}$ make $D_1(m,r)$ vanish if $\ell$ is even, and they cancel $D_2(m,r)$ if $\ell$ is odd.

Finally, we have
\begin{equation}\label{dzD3} D_3(m,r):=d_\zeta((\ell+13+m)\varpi_1+r\varpi_8)=0, \:\: \forall r\geq 0, \:\: m\in\llbracket 1,16\rrbracket.\end{equation}
These terms vanish respectively because of roots $\beta_{92}$, $\beta_{89}$, $\beta_{85}$, $\beta_{82}$, $\beta_{76}$, $\beta_{71}$, $\beta_{66}$, $\beta_{59}$, $\beta_{51}$, $\beta_{46}$, $\beta_{38}$, $\beta_{31}$, $\beta_{23}$, $\beta_{16}$, $\beta_9$, and $\beta_1$.

Knowing about all these cancellations, we can now show that $T_{k+1} =-T_{\ell-k}$ for $0\leq k \leq \lfloor\frac{\ell}{2}\rfloor$.
Indeed, we have
\begin{equation}
\begin{array}{ll}
T_{\ell-k}&= \displaystyle \sum_{r=0}^{k+1} d_\zeta(r\varpi_1+(\ell-k-r)\varpi_8)\\\\&+\displaystyle \sum_{r=k+2}^{k+12} d_\zeta(r\varpi_1+(\ell-k-r)\varpi_8)\\\\&+ \displaystyle \sum_{s=k+13}^{\ell-k} d_\zeta(s\varpi_1+(\ell-k-s)\varpi_8)\\\\
&= \displaystyle \sum_{r=0}^{k+1} -d_\zeta(r\varpi_1+(k+1-r)\varpi_8)+0+0 =-T_{k+1}.
\end{array}
\end{equation}
The terms  of the third sum cancel each other out because of Formula (\ref{dz97}). Note that if $\ell\leq 13$ or if $\frac{\ell}{2}-k \leq 6$, the third sum does not appear in $T_{\ell-k}$. The second sum might be truncated depending on the values of $k$, and if $k=\frac{\ell}{2}$, it is empty, and so is the third sum. Without loss of generality, let us now suppose that $\ell \geq 13$ and that $k$ is small enough so that all the terms of the second sum appear.
The terms of the second sum are all zero because they are equal to some $D_1(m,r)$ or $D_2(m,r)$. 

For example, if $\ell$ is even, we have $k\leq \frac{\ell}{2}$, so that $\ell-k\geq \frac{\ell}{2}$, and we can write $\ell-k=\frac{\ell}{2}+a$, with $a=\frac{\ell}{2}-k\in\llbracket 0,\frac{\ell}{2}\rrbracket$. 

We get
\begin{equation}\begin{array}{ll}
S_2 &:= \displaystyle \sum_{r=k+2}^{k+12} d_\zeta(r \varpi_1+(\ell-k-r)\varpi_8)=\displaystyle \sum_{s= 2}^{12} d_\zeta((k+s) \varpi_1+(\ell-2k-s)\varpi_8)  \\&= \displaystyle \sum_{d=1}^{11} d_\zeta((\frac{\ell}{2}-a +d+1)\varpi_1+(2a-d-1)\varpi_8)\\&=\displaystyle\sum_{d=1}^{11} d_\zeta((\lceil \frac{\ell}{2}\rceil +d-a)\varpi_1+(2a-d-1)\varpi_8)\\
&= \displaystyle \sum_{n=0}^5 d_\zeta((\lceil \frac{\ell}{2}\rceil +(2n+1)-a)\varpi_1 + (2a-2n-2)\varpi_8) \\&\quad + \displaystyle \sum_{t=1}^5 d_\zeta((\lceil \frac{\ell}{2}\rceil +2t-a)\varpi_1 + (2a-2t-1)\varpi_8) \\
&= \displaystyle \sum_{m=0}^5 D_1(m,a-m-1) + D_2(m,a-m-1).\end{array}
\end{equation}
If $\ell$ is odd, the proof is analogous.

Therefore, we get
$$Q_{\ell-k}^{(1)}= \displaystyle Q_k^{(1)} + \sum_{r=k+1}^{\ell-k} T_r = Q_k^{(1)} + \underset{0}{\underbrace{(T_{k+1} + T_{\ell-k})}} + \underset{0}{\underbrace{(T_{k+2}+T_{\ell-k-1})}}+\dots $$ 
so that $Q_{\ell-k}^{(1)} = Q_k^{(1)}$ for $0\leq k \leq \lfloor\frac{\ell}{2}\rfloor.$ In particular, for $r=0$, we get \\$Q_\ell^{(1)}=Q_0^{(1)}=1$, which proves $(ii)$.

Moreover, since we already know that  $Q_k^{(1)}>0$ for $0\leq k \leq \lfloor\frac{\ell}{2}\rfloor$, we can deduce property $(iv)$, since $Q_{\ell-k}^{(1)}=Q_k^{(1)}>0$ for $0\leq k \leq \lfloor\frac{\ell}{2}\rfloor$.

We can now deduce property $(iii)$: for $k\in\llbracket \ell+1,\ell+29\rrbracket$, we have
$$Q_k^{(1)} = \displaystyle \sum_{r=0}^k T_r  =0+\displaystyle \sum_{r=\ell+2}^k T_r .$$
The remaining $T_r$ consist in terms that are either of the form $D_1(m,r)$, $D_2(m,r)$, or that cancel each other out by Formula (\ref{dz97}). Therefore, we get $Q_k^{(1)}=0$ for $k\in\llbracket \ell+1,\ell+29\rrbracket$, which proves $(iii)$.

Finally, for $k\geq 0$, since $l=\ell+30$, we have
 $$Q_{l+k}^{(1)}= \displaystyle \sum_{r=0}^{l+k} T_r = \displaystyle \sum_{r=0}^{\ell+29} T_r  + \displaystyle \sum_{r=0}^{k} T_{l+r} =\displaystyle \sum_{r=0}^{k} T_{l+r} .$$
Moreover, for $r\geq 0$, we have
$$\begin{array}{ll}T_{l+r}  &= \displaystyle \sum_{d=0}^{l+r} d_\zeta(d\varpi_1 + (l+r-d)\varpi_8) \\\\&= \displaystyle \sum_{d=0}^{\ell+13} d_\zeta(d\varpi_1 + (l+r-d)\varpi_8)+\displaystyle\sum_{d=\ell+14}^{l-1} d_\zeta(d\varpi_1 + (l+r-d)\varpi_8) \\\\&+ \displaystyle\sum_{d=l}^{l+r}d_\zeta(d\varpi_1 + (l+r-d)\varpi_8).\end{array}$$
The first sum's terms cancel each other out by Formula (\ref{dz97}). The second one consists in terms of the form $D_3(m,r)$, which are all zero. Finally, the translation  $t_{l\varpi_1}:\lambda \mapsto \lambda + (\ell+30)\varpi_1$, which has signature 1, gives, with Proposition \ref{aweyl}:
$$d_\zeta(s\varpi_1+r\varpi_8  )=d_\zeta((l+s)\varpi_1+r\varpi_8),$$
and we get
$$T_{l+r}  = 0+0+ \displaystyle \sum_{d=0}^r d_\zeta((l+d)\varpi_1 + (r-d)\varpi_8) = \displaystyle \sum_{d=0}^r d_\zeta(d\varpi_1 + (r-d)\varpi_8) = T_r .$$
Therefore, the sequence $(T_k)$ is $l$-periodic, and so is the sequence  $(Q_k^{(1)})$, which proves $(i)$ and finishes our proof of Lemma  \ref{e81}.

Finally, by Proposition~\ref{Qk+}, we have $T_k >0$ for $k\in \llbracket 0,\lfloor \ell/2\rfloor-1\rrbracket$,
hence $Q_k^{(i)} < Q_{k+1}^{(i)}$ for $k\in \llbracket 0,\lfloor \ell/2\rfloor-1\rrbracket$.
\hfill $\Box$

\end{proof}

We can now deduce a partial result at node 3.

\begin{lemma}
Properties $(i)$ to $(iii)$ of Lemma \ref{e88} are also true at node \:\nolinebreak 3.
\end{lemma}

\begin{proof}
We use the $Q$-system at node 1:
$$ (Q_k^{(1)})^2 = Q_{k-1}^{(1)} Q_{k+1}^{(1)} + Q_k^{(3)}$$
to deduce all wanted properties.  \hfill $\Box$

\end{proof}

Node 7 is the unique neighbour of node 8.

\begin{lemma}
Properties $(i)$ to $(iv)$ of Lemma \ref{e88} also hold at node 7.
\end{lemma}

\begin{proof}
We use the $Q$-system at node 8:
$$(Q_k^{(8)})^2=Q_{k-1}^{(8)} Q_{k+1}^{(8)}+Q_k^{(7)}.$$
We deduce properties $(i)$ to $(iii)$ from this equation. Property $(iv)$ comes from the strict log-concavity of $Q_k^{(8)}$ and the fact that $Q_0^{(7)}=Q_\ell^{(7)}=1$. \hfill $\Box$
\end{proof}

We can go further along this chain of nodes.
 
\begin{lemma}\label{e86}
Properties $(i)$ to $(iii)$ of Lemma \ref{e88} also hold at nodes  6, 5, and \nolinebreak 4.
\end{lemma}

\begin{proof}
At node 6, we use the $Q$-system at node 7:
$$(Q_k^{(7)})^2=Q_{k-1}^{(7)} Q_{k+1}^{(7)} + Q_k^{(6)} Q_k^{(8)},$$
to prove property $(ii)$, as well as the $l$-periodicity when $Q_k^{(8)} \neq 0$.

For the full periodicity and property $(iii)$, we use the $Q$-system at node 6. The proof is the same as that of Lemma \ref{e745zero}.

The proof at nodes 5 and 4 is analogous: we use the $Q$-system respectively at nodes 6 and 5, and 4 and 3. 
\hfill $\Box$
\end{proof}

\begin{lemma}
At node 2, we have $Q_l^{(2)}=Q_\ell^{(2)}=1$, and $Q_{nl+\ell+k}^{(2)}=0$ for all $1\leq k \leq h-1$ and $n\in\mathbb{N}$.
\end{lemma}

\begin{proof}
The $Q$-system at node 4 gives $Q_l^{(2)}=Q_\ell^{(2)}=1$. We then use the \\$Q$-system at node 2 to obtain the zero's and their $l$-periodicity. \hfill$\Box$
\end{proof}

This completes the proof of Theorem \ref{knsE678} in type $E_8$.

\newpage

\appendix
\section*{Appendix A:  Positive roots of $E_7$}

The 63 positive roots $\beta$ are listed below.
The second column gives the height $(\rho|\beta)$, and the third column the decomposition of $\beta$
on the basis of simple roots.

\[
{
\begin{array}{|c|c|c|}
\hline
\mbox{No.} & \mbox{height} & \beta= \displaystyle \sum_{i=1}^7 b_i \alpha_i \\
\hline
1&1&(1,0,0,0,0,0,0)\\ \hline
2&1&(0,1,0,0,0,0,0)\\\hline
3&1&(0,0,1,0,0,0,0)\\\hline
4&1&(0,0,0,1,0,0,0)\\\hline
5&1&(0,0,0,0,1,0,0)\\\hline
6&1&(0,0,0,0,0,1,0)\\\hline
7&1&(0,0,0,0,0,0,1)\\\hline
8&2&(1,0,1,0,0,0,0)\\\hline
9&2&(0,1,0,1,0,0,0)\\\hline
10&2&(0,0,1,1,0,0,0)\\\hline
11&2&(0,0,0,1,1,0,0)\\\hline
12&2&(0,0,0,0,1,1,0)\\\hline
13&2&(0,0,0,0,0,1,1)\\\hline
14&3&(1,0,1,1,0,0,0)\\\hline
15&3&(0,1,1,1,0,0,0)\\\hline
16&3&(0,1,0,1,1,0,0)\\\hline
17&3&(0,0,1,1,1,0,0)\\\hline
18&3&(0,0,0,1,1,1,0)\\\hline
19&3&(0,0,0,0,1,1,1)\\\hline
20&4&(1,1,1,1,0,0,0)\\\hline
21&4&(1,0,1,1,1,0,0)\\\hline
22&4&(0,1,1,1,1,0,0)\\\hline
23&4&(0,1,0,1,1,1,0)\\\hline
24&4&(0,0,1,1,1,1,0)\\\hline
25&4&(0,0,0,1,1,1,1)\\\hline
26&5&(1,1,1,1,1,0,0)\\\hline
27&5&(1,0,1,1,1,1,0)\\\hline
28&5&(0,1,1,2,1,0,0)\\\hline
29&5&(0,1,1,1,1,1,0)\\\hline
30&5&(0,1,0,1,1,1,1)\\\hline
31&5&(0,0,1,1,1,1,1)\\\hline
32&6&(1,1,1,2,1,0,0)\\\hline
\end{array}
}
\qquad
{
\begin{array}{|c|c|c|}
\hline
\mbox{No.} & \mbox{height} & \beta= \displaystyle \sum_{i=1}^7 b_i \alpha_i \\
\hline
33&6&(1,1,1,1,1,1,0)\\\hline
34&6&(1,0,1,1,1,1,1)\\\hline
35&6&(0,1,1,2,1,1,0)\\\hline
36&6&(0,1,1,1,1,1,1)\\\hline
37&7&(1,1,2,2,1,0,0)\\\hline
38&7&(1,1,1,2,1,1,0)\\\hline
39&7&(1,1,1,1,1,1,1)\\\hline
40&7&(0,1,1,2,2,1,0)\\\hline
41&7&(0,1,1,2,1,1,1)\\\hline
42&7&(1,1,2,2,1,1,0)\\\hline
43&8&(1,1,1,2,2,1,0)\\\hline
44&8&(1,1,1,2,1,1,1)\\\hline
45&8&(0,1,1,2,2,1,1)\\\hline
46&9&(1,1,2,2,2,1,0)\\\hline
47&9&(1,1,2,2,1,1,1)\\\hline
48&9&(1,1,1,2,2,1,1)\\\hline
49&9&(0,1,1,2,2,2,1)\\\hline
50&10&(1,1,2,3,2,1,0)\\\hline
51&10&(1,1,2,2,2,1,1)\\\hline
52&10&(1,1,1,2,2,2,1)\\\hline
53&11&(1,2,2,3,2,1,0)\\\hline
54&11&(1,1,2,3,2,1,1)\\\hline
55&11&(1,1,2,2,2,2,1)\\\hline
56&12&(1,2,2,3,2,1,1)\\\hline
57&12&(1,1,2,3,2,2,1)\\\hline
58&13&(1,2,2,3,2,2,1)\\\hline
59&13&(1,1,2,3,3,2,1)\\\hline
60&14&(1,2,2,3,3,2,1)\\\hline
61&15&(1,2,2,4,3,2,1)\\\hline
62&16&(1,2,3,4,3,2,1)\\\hline
63&17&(2,2,3,4,3,2,1)\\\hline
\end{array}
}
\]

\newpage
\section*{Appendix B: Positive roots for $E_8$}

The 120 positive roots $\beta$ are listed below. The second column gives the height $(\rho\mid\beta)$, and the third column the decomposition of $\beta$ on the basis of simple roots.

\[
{\begin{array}{|c|c|c|}
\hline
\mbox{No.} & \mbox{height} & \beta= \displaystyle \sum_{i=1}^8 b_i \alpha_i \\
\hline
1&1&(1,0,0,0,0,0,0,0 )\\\hline
2&1&(0,1,0,0,0,0,0,0 )\\\hline
3&1&(0,0,1,0,0,0,0,0 )\\\hline
4&1&(0,0,0,1,0,0,0,0 )\\\hline
5&1&(0,0,0,0,1,0,0,0 )\\\hline
6&1&(0,0,0,0,0,1,0,0 )\\\hline
7&1&(0,0,0,0,0,0,1,0 )\\\hline
8&1&(0,0,0,0,0,0,0,1 )\\\hline
9&2&(1,0,1,0,0,0,0,0 )\\\hline
10&2&(0,1,0,1,0,0,0,0 )\\\hline
11&2&(0,0,1,1,0,0,0,0 )\\\hline
12&2&(0,0,0,1,1,0,0,0 )\\\hline
13&2&(0,0,0,0,1,1,0,0 )\\\hline
14&2&(0,0,0,0,0,1,1,0 )\\\hline
15&2&(0,0,0,0,0,0,1,1 )\\\hline
16&3&(1,0,1,1,0,0,0,0 )\\\hline
17&3&(0,1,1,1,0,0,0,0 )\\\hline
18&3&(0,1,0,1,1,0,0,0 )\\\hline
19&3&(0,0,1,1,1,0,0,0 )\\\hline
20&3&(0,0,0,1,1,1,0,0 )\\\hline
21&3&(0,0,0,0,1,1,1,0 )\\\hline
22&3&(0,0,0,0,0,1,1,1 )\\\hline
23&4&(1,1,1,1,0,0,0,0 )\\\hline
24&4&(1,0,1,1,1,0,0,0 )\\\hline
25&4&(0,1,1,1,1,0,0,0 )\\\hline
26&4&(0,1,0,1,1,1,0,0 )\\\hline
27&4&(0,0,1,1,1,1,0,0 )\\\hline
28&4&(0,0,0,1,1,1,1,0 )\\\hline
29&4&(0,0,0,0,1,1,1,1 )\\\hline
30&5&(1,1,1,1,1,0,0,0 )\\\hline
\end{array}
}
\qquad
{
\begin{array}{|c|c|c|}
\hline
\mbox{No.} & \mbox{height} & \beta= \displaystyle \sum_{i=1}^8 b_i \alpha_i \\
\hline
31&5&(1,0,1,1,1,1,0,0 )\\\hline
32&5&(0,1,1,2,1,0,0,0 ) \\\hline
33&5&(0,1,1,1,1,1,0,0 ) \\\hline
34&5&(0,1,0,1,1,1,1,0 ) \\\hline
35&5&(0,0,1,1,1,1,1,0 ) \\\hline
36&5&(0,0,0,1,1,1,1,1 ) \\\hline
37&6&(1,1,1,2,1,0,0,0 ) \\\hline
38&6&(1,1,1,1,1,1,0,0 ) \\\hline
39&6&(1,0,1,1,1,1,1,0 )\\\hline
40&6&(0,1,1,2,1,1,0,0 ) \\\hline
41&6&(0,1,1,1,1,1,1,0 ) \\\hline
42&6&(0,1,0,1,1,1,1,1 ) \\\hline
43&6&(0,0,1,1,1,1,1,1 ) \\\hline
44&7&(1,1,2,2,1,0,0,0 ) \\\hline
45&7&(1,1,1,2,1,1,0,0 ) \\\hline
46&7&(1,1,1,1,1,1,1,0 ) \\\hline
47&7&(1,0,1,1,1,1,1,1 ) \\\hline
48&7&(0,1,1,2,2,1,0,0 ) \\\hline
49&7&(0,1,1,2,1,1,1,0 ) \\\hline
50&7&(0,1,1,1,1,1,1,1 ) \\\hline
51&8&(1,1,2,2,1,1,0,0 ) \\\hline
52&8&(1,1,1,2,2,1,0,0 ) \\\hline
53&8&(1,1,1,2,1,1,1,0 ) \\\hline
54&8&(1,1,1,1,1,1,1,1 ) \\\hline
55&8&(0,1,1,2,2,1,1,0 ) \\\hline
56&8&(0,1,1,2,1,1,1,1 ) \\\hline
57&9&(1,1,2,2,2,1,0,0 ) \\\hline
58&9&(1,1,2,2,1,1,1,0 ) \\\hline
59&9&(1,1,1,2,2,1,1,0 ) \\\hline
60&9&(1,1,1,2,1,1,1,1 ) \\\hline
\end{array}
}\]

\newpage

\[
{\begin{array}{|c|c|c|}
\hline
\mbox{No.} & \mbox{height} & \beta= \displaystyle \sum_{i=1}^8 b_i \alpha_i \\
\hline
61&9&(0,1,1,2,2,2,1,0 ) \\\hline
62&9&(0,1,1,2,2,1,1,1 ) \\\hline
63&10&(1,1,2,3,2,1,0,0 ) \\\hline
64&10&(1,1,2,2,2,1,1,0 ) \\\hline
65&10&(1,1,2,2,1,1,1,1 ) \\\hline
66&10&(1,1,1,2,2,2,1,0 ) \\\hline
67&10&(1,1,1,2,2,1,1,1 ) \\\hline
68&10&(0,1,1,2,2,2,1,1 ) \\\hline
69&11&(1,2,2,3,2,1,0,0 ) \\\hline
70&11&(1,1,2,3,2,1,1,0 ) \\\hline
71&11&(1,1,2,2,2,2,1,0 ) \\\hline
72&11&(1,1,2,2,2,1,1,1 ) \\\hline
73&11&(1,1,1,2,2,2,1,1 ) \\\hline
74&11&(0,1,1,2,2,2,2,1 )\\\hline
75&12&(1,2,2,3,2,1,1,0 ) \\\hline
76&12&(1,1,2,3,2,2,1,0 ) \\\hline
77&12&(1,1,2,3,2,1,1,1 ) \\\hline
78&12&(1,1,2,2,2,2,1,1 ) \\\hline
79&12&(1,1,1,2,2,2,2,1 ) \\\hline
80&13&(1,2,2,3,2,2,1,0 ) \\\hline
81&13&(1,2,2,3,2,1,1,1 ) \\\hline
82&13&(1,1,2,3,3,2,1,0 ) \\\hline
83&13&(1,1,2,3,2,2,1,1 ) \\\hline
84&13&(1,1,2,2,2,2,2,1 ) \\\hline
85&14&(1,2,2,3,3,2,1,0 ) \\\hline
86&14&(1,2,2,3,2,2,1,1 ) \\\hline
87&14&(1,1,2,3,3,2,1,1 ) \\\hline
88&14&(1,1,2,3,2,2,2,1 ) \\\hline
89&15&(1,2,2,4,3,2,1,0 ) \\\hline
90&15&(1,2,2,3,3,2,1,1 ) \\\hline
\end{array}
}
\qquad
{
\begin{array}{|c|c|c|}
\hline
\mbox{No.} & \mbox{height} & \beta= \displaystyle \sum_{i=1}^8 b_i \alpha_i \\
\hline
91&15&(1,2,2,3,2,2,2,1 ) \\\hline
92&15&(1,1,2,3,3,2,2,1 ) \\\hline
93&16&(1,2,3,4,3,2,1,0 ) \\\hline
94&16&(1,2,2,4,3,2,1,1 ) \\\hline
95&16&(1,2,2,3,3,2,2,1 ) \\\hline
96&16&(1,1,2,3,3,3,2,1 ) \\\hline
97&17&(2,2,3,4,3,2,1,0 ) \\\hline
98&17&(1,2,3,4,3,2,1,1 ) \\\hline
99&17&(1,2,2,4,3,2,2,1 ) \\\hline
100&17&(1,2,2,3,3,3,2,1 ) \\\hline
101&18&(2,2,3,4,3,2,1,1 ) \\\hline
102&18&(1,2,3,4,3,2,2,1 ) \\\hline
103&18&(1,2,2,4,3,3,2,1 ) \\\hline
104&19&(2,2,3,4,3,2,2,1 ) \\\hline
105&19&(1,2,3,4,3,3,2,1 ) \\\hline
106&19&(1,2,2,4,4,3,2,1 ) \\\hline
107&20&(2,2,3,4,3,3,2,1 ) \\\hline
108&20&(1,2,3,4,4,3,2,1 ) \\\hline
109&21&(2,2,3,4,4,3,2,1 ) \\\hline
110&21&(1,2,3,5,4,3,2,1 ) \\\hline
111&22&(2,2,3,5,4,3,2,1 ) \\\hline
112&22&(1,3,3,5,4,3,2,1 ) \\\hline
113&23&(2,3,3,5,4,3,2,1 ) \\\hline
114&23&(2,2,4,5,4,3,2,1 )\\\hline
115&24&(2,3,4,5,4,3,2,1 ) \\\hline
116&25&(2,3,4,6,4,3,2,1 ) \\\hline
117&26&(2,3,4,6,5,3,2,1 ) \\\hline
118&27&(2,3,4,6,5,4,2,1 ) \\\hline
119&28&(2,3,4,6,5,4,3,1 ) \\\hline
120&29&(2,3,4,6,5,4,3,2 ) \\
\hline
\end{array}
}\]

\end{document}